\documentclass{article}
\usepackage{amsmath,amssymb,amsthm}
\usepackage[colorlinks,citecolor=blue]{hyperref}
\usepackage[utf8]{inputenc}
\usepackage{tikz}
\usepackage[all]{xy}

\newtheorem{theorem}{Theorem}[section] 
\newtheorem{proposition}[theorem]{Proposition}

\newtheorem{lemma}[theorem]{Lemma}

\theoremstyle{definition}

\newtheorem{remark}[theorem]{Remark}

\newenvironment{psmallmatrix}
  {\left(\begin{smallmatrix}}
  {\end{smallmatrix}\right)}

\newcommand{\TA}{\mathbb{A}}
\newcommand{\TB}{\mathbb{B}}

\newcommand{\TD}{\mathbb{D}}

\newcommand{\TyH}{\mathbb{H}}
\newcommand{\TI}{\mathbb{I}}

\newcommand{\kk}{\mathbf{k}}

\newcommand{\cam}{\operatorname{Cambrian}}
\newcommand{\nest}{\operatorname{NonNesting}}

\title{On the wildness of cambrian lattices}
\author{Fr\'ed\'eric Chapoton\footnote{Cet auteur a b\'en\'efici\'e d'une aide de l'Agence Nationale de la Recherche (projet Carma, r\'ef\'erence ANR-12-BS01-0017).}, Baptiste Rognerud\footnote{Cet auteur a b\'en\'efici\'e d'un financement de l'IDEX BMM/PN/AM/N\textsuperscript{o} 2016-096c.}}

\date{\today}

\begin{document}

\maketitle


\begin{abstract}
  In this note, we investigate the representation type of the cambrian
  lattices and some other related lattices. The result is expressed as
  a very simple trichotomy. When the rank of the underlined Coxeter
  group is at most $2$, the lattices are of finite representation
  type. When the Coxeter group is a reducible group of type $\TA_1^3$,
  the lattices are of tame representation type. In all the other
  cases they are of wild representation type.
\end{abstract}

\section{Introduction}

By a famous result of Y. Drozd \cite{drozd}, associative algebras can be
classified according to their representation type, which can be
finite, tame or wild. Algebras of finite representation type have a
finite number of isomorphism classes of indecomposable modules. They
are quite rare in some sense, and a random algebra, meaning roughly any
sufficiently complicated algebra, should be expected \textit{a priori} to be
wild, unless there is a specific reason for the contrary. Algebras of
tame type are living on the frontier between finite type and wild
type, and should not be expected to appear frequently either.

Given any finite poset $P$ and a base field $\kk$, one can define the
incidence algebra of $P$ over $\kk$, which is a finite dimensional
associative algebra of finite global dimension. It is a natural
question to ask what is the representation type of incidence
algebras. A general description of the representation-finite case has
been made in \cite{loupias1, loupias2}. The tame case has been
considered in \cite{les1,les2}.

In this short note, an answer is given for some families of posets,
namely the cambrian lattices and some other related lattices. Cambrian
lattices have been introduced by N. Reading \cite{reading1, reading2}
to provide a description of the most combinatorial aspects of the
theory of cluster algebras of S. Fomin and A. Zelevinsky. More than
that, these posets allow to extend this combinatorial description from
finite root systems to finite Coxeter groups, therefore encompassing
the non-crystallographic finite Coxeter groups of type $\TI$ and $\TyH$.

Our main aim in this note is to tell which of the cambrian lattices have
finite representation type, and which have tame or wild representation type. It turns out that for this family of posets the trichotomy is particularly simple: when the rank of $W$ is $1$ or $2$, then the cambrian lattices are of finite representation type. Otherwise they are of wild representation unless they are isomorphic to the cube which is tame. This last case arises exactly when $W = \TA_1\times \TA_1\times \TA_1$. 

In the remaining part of the note, the same result is obtained for two
related families, namely lattices of order ideals of root posets and
Stokes lattices.

\section{Representation theory of posets}

The incidence algebra of a poset $P$ over a field $\kk$ has a basis
$I_{a,b}$ indexed by pairs of comparable elements $a \leq b$ in $P$
and its associative product is defined on this basis by
\begin{equation}
  I_{a,b} I_{c,d} = 
  \begin{cases}
    0 &\text{if}\quad b \not= c,\\
    I_{a,d} &\text{if}\quad b = c.
  \end{cases}
\end{equation}

Every poset can be seen as a small category, where there is a (unique)
map from $a$ to $b$ if and only if $a \leq b$ in $P$. The category of
finite-dimensional modules over the incidence algebra of $P$ over
$\kk$ is equivalent to the category of functors from the poset $P$ to
the category of finite-dimensional vector fields over $\kk$, denoted by $\mathcal{F}_{P,\kk}$.

This can also be restated, using the Hasse diagram of the poset $P$,
as the category of representations of a quiver with relations. The
arrows $a \to b$ of the quiver are the cover relations $a \triangleleft b$ in
the poset $P$, and the relations are commuting relations, saying that
any two paths sharing their ends are equal.

By a slight abuse of notation, we will say that two posets $P$ and $Q$
are derived equivalent when the derived categories of modules over
their incidence algebras are triangle-equivalent.

Similarly, we will say that $P$ has finite, tame or wild representation
type if the incidence algebra of $P$ has this property.

In the rest of this note we assume that $\kk$ is an algebraically closed field. We will systematically use the following well-known Lemma. Recall that a subposet of a poset $Y$ is the induced partial order on a subset of $Y$. 
\begin{lemma}
\label{sub_poset}
Let $X$ be a subposet of a finite poset $Y$. 
\begin{enumerate}
\item If $X$ is of wild representation type, then $Y$ is of wild representation type.
\item If $X$ is of tame representation type, then $Y$ is either of tame or of wild representation type.
\end{enumerate} 
\end{lemma}
\begin{proof}
We view $X$ and $Y$ as finite categories. Then, the embedding $i$ of $X$ in $Y$ is a fully-faithful functor. The precomposition by this functor induces a functor $i^{-1}$ from $\mathcal{F}_{Y,k}$ to $\mathcal{F}_{X,\kk}$. By usual arguments, it has a left adjoint $i_{!}$ given by the left Kan extension along $i$. If $F\in \mathcal{F}_{X,\kk}$, then by Proposition $3.7.3$ of \cite{borceux}, we have $(i^{-1} \circ i_{!})(F) \cong F$. It follows that the functor $i_{!}$ sends an indecomposable object of $\mathcal{F}_{X,\kk}$ to an indecomposable object of $\mathcal{F}_{Y,\kk}$. Moreover, for $F_1$ and $F_2$ in $\mathcal{F}_{X,\kk}$, we have $i_{!}(F_1)\cong i_{!}(F_2)$ if and only if $F_1\cong F_2$. 
\end{proof}
%
We will also need the following well-known result.
\begin{proposition}\label{contraction}
Let $X$ and $Y$ be two finite posets and $f : X\to Y$ be a surjective morphism of posets such that $f^{-1}(y)$ is connected for every $y\in Y$. Then, if $Y$ is wild representation type, so is $X$.
\end{proposition}
\begin{proof}
This is Proposition $1.3$ of \cite{loupias2}. This follows from the fact that the functor $f^{-1} : \mathcal{F}_{Y,\kk} \to \mathcal{F}_{X,\kk}$ is a fully-faithful embedding. 
\end{proof}
Derived equivalences provide powerful tools for the study of finite dimensional algebras and finite posets. When two algebras share the same derived categories, they share a number of invariants such as the number of isomorphism classes of simple modules, or the center. However, the representation type is not preserved by derived equivalences. Still, when a finite dimensional algebra is derived equivalent to a hereditary algebra, its representation type is dominated\footnote{Here dominated means smaller in the order where (finite) $\leq$ (tame) $\leq$ (wild).} by the representation type of the hereditary algebra. 
\begin{proposition}
  \label{domine}
  Let $A$ be a finite dimensional $\kk$-algebra. Let $\mathcal{H}$ be an abelian hereditary category. If $D^b(A) \cong D^b(\mathcal{H})$, then the representation type of $A$ is dominated by the representation type of $\mathcal{H}$.  
\end{proposition}
\begin{proof}
Let $F : D^{b}(A) \to D^{b}(\mathcal{H})$ be an equivalence of triangulated categories. As usual, we view the indecomposable $A$-modules as complexes concentrated in degree zero. Then, the functor $F$ sends an indecomposable $A$-module to an indecomposable object of $D^{b}(\mathcal{H})$. Moreover, there exist two integers $m$ and $n$, independent of $X$ (but which depend on the choice of the equivalence $F$) such that the homology of $F(X)$ is concentrated in degrees in $[m,n]$. Since the category $\mathcal{H}$ is hereditary, the indecomposable objects of $D^b(\mathcal{H})$ are just shifts of indecomposable objects of $\mathcal{H}$. So the images of the indecomposable $A$-modules under $F$ lie in a finite number of copies of $\mathcal{H}$. For more details, see for example Section $2$ of \cite{happel_zacharia}. 
\end{proof}

\begin{proposition}[Ladkani]\label{flip_flop}
Let $X$ be a finite poset with a unique maximal element $\hat{1}$. Let $Y = X\backslash \{\hat{1}\}\sqcup \{\hat{0}\}$ where $\hat{0}$ is the unique minimal element of $Y$. Then, the posets $X$ and $Y$ are derived equivalent. 
\end{proposition}
\begin{proof}
This is a special case of the much more general \emph{flip-flop} of Ladkani (see Corollary $1.3$ of \cite{ladkani_universal}). In this simple case, the derived equivalence can be realized by the following tilting complex. If $x\in X$, we denote by $P_x$ the projective indecomposable module corresponding to the element $x$. If $x\neq \hat{1}$, we consider the complex $T_{x} = P_{\hat{1}} \overset{f}{\to} P_{x}$, where $f$ is the canonical embedding of $P_{\hat{1}}$ into $P_{x}$ and $P_{\hat{1}}$ is in degree zero of the complex. Then, $T = P_{\hat{1}} \bigoplus_{\hat{1}\neq x\in X} T_{x}$ is a tilting complex and its endomorphism algebra is the incidence algebra of $Y$. 
\end{proof}

\begin{proposition}\label{wild_subposet}
  Let $\kk$ be an algebraically closed field. Let $n\geqslant 3$. Let
  $Y$ be a finite poset whose Hasse diagram is a non circular
  orientation of an affine diagram of type $\TA_n^{(1)}$. Let
  $X = Y\sqcup \{\omega\}$ be a poset whose Hasse diagram contains
  exactly two more edges, in such a way that $\omega$ is the maximal
  element or the minimal element of a commutative square. Then, $X$ is
  of wild representation type.
\end{proposition}
\begin{proof}
Let us remark that it is always possible to produce such a poset $X$. The orientation of the diagram $\TA_n^{(1)}$ is non circular, so there are subposets of $\TA_n^{(1)}$ of the form $\overset{a}{\bullet} \rightarrow \bullet \leftarrow \overset{b}{\bullet}$ and of the form $\overset{c}{\bullet}  \leftarrow \bullet \rightarrow \overset{d}{\bullet} $. If we set $\omega>a$ and $\omega > b$ we have a poset $X$ such that $\omega$ is the maximal element of a commutative square. If we set $\omega<c$ and $\omega <d$, then $\omega$ is the minimal element of a commutative square. By duality, it is enough to prove the result when $\omega$ is the minimal element of a commutative square. We exhibit a two-parameter family of pairwise non-isomorphic indecomposable representations for the poset $X$ inspired by a similar construction for the preprojective algebra of type $A_6$ (see \cite{mo}). This implies that $X$ is of wild representation type. 
\newline Let $i : \TA_n^{(1)} \to X$ be the canonical embedding. By precomposition, we have a functor $i^{-1} : \mathcal{F}_{X,\kk} \to \mathcal{F}_{\TA_n^{(1)},\kk}$ which is nothing but the obvious restriction to the subposet $\TA_n^{(1)}$. 
\newline We fix $\alpha$ an arrow of the Hasse diagram of $X$ which is not in the commutative square. Let $\lambda \in \kk$. For $x\in \TA_n^{(1)}$, we set $M(\lambda)(x) = \kk$. We let $M(\lambda)(\alpha) : \kk\to \kk$ be the multiplication by $\lambda$. And for every arrow $\alpha'$ in the Hasse diagram of $\TA_n^{(1)}$, we let $M(\lambda)(\alpha')=\operatorname{Id}_{\kk}$. Extending $M(\lambda)$ to any morphism of the finite category $\TA_n^{(1)}$, we have a functor $M(\lambda)$ in $\mathcal{F}_{\TA_n^{(1)},\kk}$. For $\lambda, \mu \in \kk$, it is easy to see that
\[ \operatorname{Hom}_{\mathcal{F}_{\TA_n^{(1)},\kk}}\big(M(\lambda),M(\mu)\big)\cong \left\{
      \begin{aligned}
        k & \hbox{ if }  \lambda = \mu, \\
        0 & \hbox{ otherwise.}
      \end{aligned}
    \right. \]
Here, the isomorphism $\kk \cong \operatorname{End}(M(\lambda))$ is given by the application that sends $t\in \kk$ to the natural transformation $(t_x)_{x\in \TA_n^{(1)}}\in \operatorname{End}(M(\lambda))$, where $t_{x} : M(\lambda)(x) =\kk \to M(\lambda)(x) = \kk$ is the multiplication by $t$. Since the functor $M(\lambda)$ has a local endomorphism algebra, it is indecomposable. 
\newline\indent Let $\lambda,\mu \in \kk$ such that $\lambda \neq \mu$. Then, we consider the functor $M(\lambda,\mu)$ in $\mathcal{F}_{X,\kk}$ defined as follows: 
\[ i^{-1}\big(M(\lambda,\mu)\big)=M(\lambda)\oplus M(\mu), \hbox{ and } M(\lambda,\mu)(\omega) = \kk. \]
For the two arrows $g_1$ and $g_2$ starting at $\omega$, we let $M(\lambda,\mu)(g_i)$ be the diagonal embedding $\bigtriangleup$ of $\kk$ into $\kk\oplus \kk$. In other terms, the representation $M(\lambda,\mu)$ has the following shape:
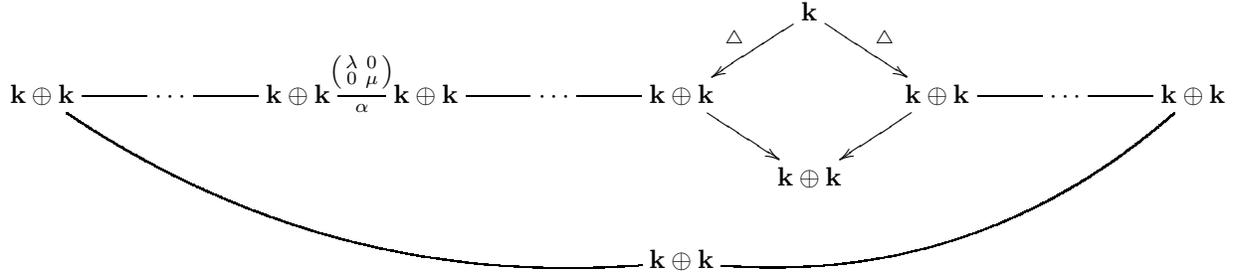
\[
\xymatrix @!0 @R=11mm @C=1,7cm {
&&&&&& \kk\ar[rd]^{\bigtriangleup}\ar[ld]_{\bigtriangleup} \\
\kk\oplus\kk \ar@{-}@/_2pc/[ddrrrrr]\ar@{-}[r]& \cdots \ar@{-}[r]& \kk\oplus\kk\ar@{-}[r]^{\begin{psmallmatrix}\lambda & 0\\0 & \mu \end{psmallmatrix}}_{\alpha} & \kk\oplus\kk\ar@{-}[r] & \cdots\ar@{-}[r] & \kk\oplus\kk\ar[rd] & & \kk\oplus\kk\ar[ld]\ar@{-}[r] & \cdots \ar@{-}[r]& \kk\oplus\kk \ar@{-}@/^2pc/[ddllll] \\ 
&&&&&& \kk\oplus\kk \\
&&&&& \kk\oplus\kk
}
\] 
where all the non-labelled edges correspond to the identity matrix. 
\newline Let $f$ be an endomorphism of $M(\lambda,\mu)$. Then, $i^{-1}(f)$ is an endomorphism of $M(\lambda)\oplus M(\mu)$. In particular, because one assumes that $\lambda \not= \mu$, it is of the form $f_{1}\oplus f_{2}$ where $f_1$ is an endomorphism of $M(\lambda)$ and $f_2$ is an endomorphism of $M(\mu)$. Moreover, we know that $f_1$ (resp. $f_2$) is the natural transformation given by the multiplication by an element $t_1 \in \kk$ (resp. $t_2 \in \kk$). It follows that for every $x\in \TA_n^{(1)}$, we have $f_{x} =\begin{psmallmatrix}t_1 & 0\\0 & t_2\end{psmallmatrix}$. If we look at the situation at the vertex $\omega$, we have the following commutative square:
\[ 
\xymatrix{
\kk \ar[r]^{f_{\omega}} \ar[d]^{\bigtriangleup} & \kk\ar[d]^{\bigtriangleup} \\
\kk\oplus\kk \ar[r]^{f_1\oplus f_2} &\kk\oplus\kk.
}
\]
So, we have $t_1 = t_2 = f_{\omega}(1)$. It is now clear that $\operatorname{End}\big(M(\lambda,\mu)\big)\cong \kk$. In particular, the functor $M(\lambda,\mu)$ is indecomposable. Let $\lambda'\neq\mu' \in \kk$. If there is an isomorphism $f : M(\lambda,\mu) \to M(\lambda',\mu')$, then by restriction, there is an isomorphism from $M(\lambda)\oplus M(\mu)$ to $M(\lambda')\oplus M(\mu')$. This implies that $\{\lambda,\mu\} = \{\lambda',\mu'\}$. Since the roles of $\lambda$ and $\mu$ are symmetric, there is an isomorphism $M(\lambda,\mu)\cong M(\mu,\lambda)$. In other words, there is an isomorphism between the indecomposable representations $M(\lambda,\mu)$ and $M(\lambda',\mu')$ if and only if $\{\lambda,\mu\} = \{\lambda',\mu'\}$. 
\end{proof}





\section{Cambrian lattices}

Let $W$ be a finite Coxeter group. Let $n$ denote the rank of $W$,
which is the number of simple generators $\{s_1,\dots,s_n\}$.

When the Coxeter graph of $W$ is connected, so that $W$ is not a
product of smaller Coxeter groups, we will say that $W$ is
irreducible, and reducible otherwise.

Recall that a Coxeter element $c$ in $W$ is the product of the simple
generators in some order.

Once a Coxeter element $c$ is chosen, one can define the cambrian
lattice $\cam(W,c)$ using restriction of the weak order on $W$ to the
set of $c$-sortable elements. The reader is referred to the original
articles \cite{reading1, reading2} for the precise definition.

For a given $W$, the cambrian lattices for various choices of $c$
share the same number of elements, which is called the Coxeter-Catalan
number of type $W$. Much more is true, namely their Hasse diagrams are
orientations of the same unoriented $n$-regular graph, which is the
exchange graph of cluster theory when $W$ is crystallographic.

Something deeper also holds, at least in the case of simply-laced
crystallographic $W$. S. Ladkani has proved in \cite{ladkani_cluster},
in the much more general case of quivers, the following statement.

\begin{proposition}
  Let $W$ be a simply-laced crystallographic root-system. Let $c$ and
  $c'$ be two Coxeter elements for $W$. Then the incidence algebras of
  $\cam(W,c)$ and $\cam(W,c')$ are derived equivalent.
\end{proposition}

This proposition is expected to hold in all cases, but has not
been proved yet, to the best of our knowledge.

\subsection{Easy types}

There are two situations when one can easily determine the
representation type of the poset $\cam(W,c)$.

\begin{lemma}
  If the rank $n$ of $W$ is at most $2$, every poset $\cam(W,c)$ is of
  finite representation type.
\end{lemma}
\begin{proof}
  In rank $1$, this is trivial for the quiver with one vertex.

  In rank $2$, for the Coxeter groups of type $\TI_2(h)$, the Hasse
  diagram of $\cam(W,c)$ is obtained by adding a minimal element and a
  maximal element to the disjoint union of the chain poset of size
  $h - 1$ and the poset with one element.

  By Proposition \ref{flip_flop}, the poset $\cam(W,c)$ is derived equivalent to a
  poset which is just a quiver, and has type $\TD_{h+2}$. This quiver
  is therefore of finite representation type, and the conclusion
  follows by Proposition \ref{domine}.
\end{proof}
\smallskip

If the rank $n$ is at least $4$, then one can proceed as
follows. Recall that the cambrian lattice of a Coxeter group $W$ of rank $n$
has a $n$-regular Hasse diagram.

\begin{lemma}
  \label{rank4}
  Let $P$ be a lattice whose Hasse diagram is a $4$-regular
  graph. Then $P$ is of wild representation type.
\end{lemma}
\begin{proof}
  Let $a$ be the unique minimal element of $P$ and let
  $b_1,b_2,b_3,b_4$ be the elements that cover $a$.

  Assume first that some $b_i$ is covered by an element $c$ that does
  not cover any other $b_j$. The induced quiver on the set
  $\{a,b_1,b_2,b_3,b_4\}$ has affine type $\TD_4^{(1)}$. Then, the
  induced quiver on the set $\{a,b_1,b_2,b_3,b_4,c\}$ has wild
  representation type.

  Otherwise, every element covering one of the $b_i$ covers at least
  another $b_j$. Because $P$ is a lattice, for every pair
  $1 \leq i < j \leq 4$, there is at most one element $c_{i,j}$
  covering both $b_i$ and $b_j$. So in particular there are at most
  $6$ elements covering one of the $b_i$. Using the hypothesis that
  the Hasse diagram is $4$-regular, one can deduce that, for every
  $1\leq i < j\leq 4$, there exists an element $c_{i,j}$ covering
  exactly both $b_i$ and $b_j$. Note that there can be no arrows
  between the $c_{i,j}$.

  Now consider the subquiver on the set
  $\{b_1,b_2,b_3,c_{1,2},c_{2,3},c_{1,3},c_{1,4}\}$. Without $c_{1,4}$,
  this would be an affine quiver of type $\TA_5^{(1)}$. This is therefore a
  quiver of wild representation type.
\end{proof}

\begin{lemma}
  \label{rank5}
  Let $P$ be a poset that has a minimal element $a$ covered by $n\geq
  5$ elements $b_1,\dots,b_n$. Then $P$ is of wild representation type.
\end{lemma}
\begin{proof}
  The induced quiver on $\{a,b_1,\dots,b_n\}$ has wild
  representation type because $n \geq 5$.
\end{proof}

In the case of rank $n=3$, the situation is less clear. By picking
appropriate subsets of vertices, we will now show that the posets
$\cam(W,c)$ are wild unless they are isomorphic to the cube. 

\subsection{Reducible types of rank $3$}

\begin{lemma}[Lenzing]
  \label{cube}
  The cube poset is of tame representation type.
\end{lemma}
\begin{proof}
The subposet of the cube consisting of all the vertices except the minimal and the maximal elements of the cube is a quiver of type $\TA_5^{(1)}$. By Lemma \ref{sub_poset}, the cube has an infinite representation type. Moreover, by Example $18.6.2$ of \cite{lenzing_coxeter} or by Example $2.2$ of \cite{ladkani_canonical}, the incidence algebra of the cube is derived equivalent to the weighted projective line $\mathbb{X}(3,3,3)$. By Section $4$ of \cite{lenzing-meltzer}, the category $\operatorname{coh}(\mathbb{X}(3,3,3))$ is tame. Since it is a hereditary category, the result follows from Proposition \ref{domine}. 
\end{proof}
%
%

\begin{proposition}
  \label{prisme_pentagonal}
  Let $h\geqslant 3$. Then, the poset $\cam(W_{\TA_1 \TI_2(h)},c)$ is wild.
\end{proposition}
\begin{proof}
The poset attached to the reducible type $\TA_1 \TI_2(h)$ is the cartesian product of a segment and a $(h-1,1)$-cycle:  a poset obtained  by adding a minimal element and a maximal element to the disjoint union of the chain poset of size  $h - 1$ and the poset with one element. We exhibit a subposet which has wild representation type by Proposition \ref{wild_subposet}.

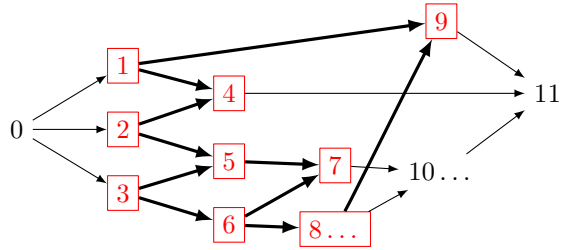
\begin{center}
\begin{tikzpicture}[rotate=270,>=latex,line join=bevel,scale=0.8]
  \node (node_0) at (57.0bp,7.0bp) [draw,draw=none] {$0$};
  \node (node_1) at (27.0bp,57.0bp) [draw,color=red] {$1$};
  \node (node_10) at (77.0bp,207.0bp) [draw,draw=none] {$10\dots$};
  \node (node_11) at (40.0bp,257.0bp) [draw,draw=none] {$11$};
  \node (node_2) at (87.0bp,57.0bp) [draw,color=red] {$3$};
  \node (node_3) at (102.0bp,107.0bp) [draw,color=red] {$6$};
  \node (node_4) at (104.0bp,157.0bp) [draw,color=red] {$8\dots$};
  \node (node_5) at (6.0bp,207.0bp) [draw,color=red] {$9$};
  \node (node_6) at (57.0bp,57.0bp) [draw,color=red] {$2$};
  \node (node_7) at (40.0bp,107.0bp) [draw,color=red] {$4$};
  \node (node_8) at (72.0bp,107.0bp) [draw,color=red] {$5$};
  \node (node_9) at (74.0bp,157.0bp) [draw,color=red] {$7$};
  \draw [black,->] (node_7) ..controls (40.0bp,134.41bp) and (40.0bp,207.41bp)  .. (node_11);
  \draw [black,->,very thick] (node_2) ..controls (90.778bp,70.09bp) and (94.308bp,81.386bp)  .. (node_3);
  \draw [black,->] (node_0) ..controls (64.691bp,20.305bp) and (72.08bp,32.128bp)  .. (node_2);
  \draw [black,->] (node_5) ..controls (14.767bp,220.38bp) and (23.267bp,232.38bp)  .. (node_11);
  \draw [black,->,very thick] (node_2) ..controls (83.222bp,70.09bp) and (79.692bp,81.386bp)  .. (node_8);
  \draw [black,->,very thick] (node_6) ..controls (52.718bp,70.09bp) and (48.717bp,81.386bp)  .. (node_7);
  \draw [black,->] (node_10) ..controls (67.405bp,220.45bp) and (58.016bp,232.63bp)  .. (node_11);
  \draw [black,->,very thick] (node_3) ..controls (94.822bp,120.31bp) and (87.925bp,132.13bp)  .. (node_9);
  \draw [black,->,very thick] (node_1) ..controls (23.24bp,84.5bp) and (12.803bp,158.06bp)  .. (node_5);
  \draw [black,->,very thick] (node_8) ..controls (72.498bp,119.95bp) and (72.954bp,130.9bp)  .. (node_9);
  \draw [black,->,very thick] (node_3) ..controls (102.5bp,119.95bp) and (102.95bp,130.9bp)  .. (node_4);
  \draw [black,->] (node_0) ..controls (57.0bp,19.947bp) and (57.0bp,30.897bp)  .. (node_6);
  \draw [black,->] (node_4) ..controls (97.119bp,170.23bp) and (90.568bp,181.88bp)  .. (node_10);
  \draw [black,->] (node_0) ..controls (49.309bp,20.305bp) and (41.92bp,32.128bp)  .. (node_1);
  \draw [black,->,very thick] (node_6) ..controls (60.778bp,70.09bp) and (64.308bp,81.386bp)  .. (node_8);
  \draw [black,->] (node_9) ..controls (74.747bp,169.95bp) and (75.431bp,180.9bp)  .. (node_10);
  \draw [black,->,very thick] (node_1) ..controls (30.255bp,70.019bp) and (33.267bp,81.141bp)  .. (node_7);
  \draw [black,->,very thick] (node_4) ..controls (82.745bp,168.41bp) and (43.289bp,187.74bp)  .. (node_5);
\end{tikzpicture}
\end{center}
\end{proof}

\subsection{Type $\TA_3$}

Let us show that the Tamari lattice with $14$ vertices has a wild representation type by exhibiting a wild subquiver.

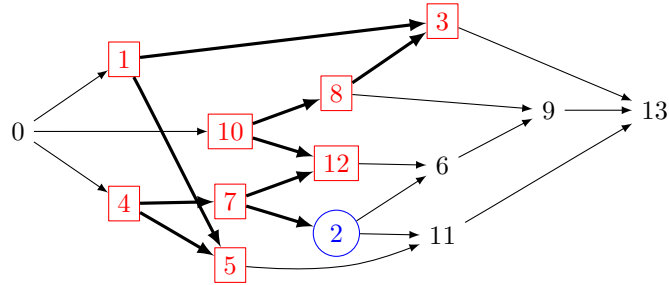
\begin{center}
\begin{tikzpicture}[rotate=270,>=latex,line join=bevel,scale=0.8]
  \node (node_0) at (58.0bp,7.0bp) [draw,draw=none] {$0$};
  \node (node_1) at (24.0bp,57.0bp) [draw,color=red] {$1$};
  \node (node_10) at (58.0bp,107.0bp) [draw,color=red] {$10$};
  \node (node_11) at (107.0bp,207.0bp) [draw,draw=none] {$11$};
  \node (node_12) at (73.0bp,157.0bp) [draw,color=red] {$12$};
  \node (node_13) at (48.0bp,307.0bp) [draw,draw=none] {$13$};
  \node (node_2) at (106.0bp,157.0bp) [draw,color=blue,shape=circle] {$2$};
  \node (node_3) at (6.0bp,207.0bp) [draw,color=red] {$3$};
  \node (node_4) at (92.0bp,57.0bp) [draw,color=red] {$4$};
  \node (node_5) at (121.0bp,107.0bp) [draw,color=red] {$5$};
  \node (node_6) at (74.0bp,207.0bp) [draw,draw=none] {$6$};
  \node (node_7) at (91.0bp,107.0bp) [draw,color=red] {$7$};
  \node (node_8) at (40.0bp,157.0bp) [draw,color=red] {$8$};
  \node (node_9) at (48.0bp,257.0bp) [draw,draw=none] {$9$};
  \draw [black,->] (node_12) ..controls (73.249bp,169.95bp) and (73.477bp,180.9bp)  .. (node_6);
  \draw [black,->,very thick] (node_4) ..controls (91.751bp,69.947bp) and (91.523bp,80.897bp)  .. (node_7);
  \draw [black,->,very thick] (node_4) ..controls (99.434bp,70.305bp) and (106.58bp,82.128bp)  .. (node_5);
  \draw [black,->,very thick] (node_1) ..controls (45.038bp,68.411bp) and (84.091bp,87.736bp)  .. (node_5);
  \draw [black,->,very thick] (node_7) ..controls (94.778bp,120.09bp) and (98.308bp,131.39bp)  .. (node_2);
  \draw [black,->,very thick] (node_1) ..controls (20.788bp,84.409bp) and (11.909bp,157.41bp)  .. (node_3);
  \draw [black,->] (node_2) ..controls (106.25bp,169.95bp) and (106.48bp,180.9bp)  .. (node_11);
  \draw [black,->] (node_11) ..controls (94.696bp,228.44bp) and (69.902bp,269.62bp)  .. (node_13);
  \draw [black,->] (node_8) ..controls (41.641bp,178.1bp) and (44.88bp,217.78bp)  .. (node_9);
  \draw [black,->] (node_3) ..controls (14.702bp,228.3bp) and (32.092bp,268.88bp)  .. (node_13);
  \draw [black,->,very thick] (node_7) ..controls (86.466bp,120.09bp) and (82.23bp,131.39bp)  .. (node_12);
  \draw [black,->] (node_0) ..controls (66.767bp,20.377bp) and (75.267bp,32.377bp)  .. (node_4);
  \draw [black,->] (node_5) ..controls (122.53bp,123.75bp) and (123.96bp,145.77bp)  .. (121.0bp,164.0bp) .. controls (119.51bp,173.18bp) and (116.3bp,183.06bp)  .. (node_11);
  \draw [black,->,very thick] (node_10) ..controls (53.466bp,120.09bp) and (49.23bp,131.39bp)  .. (node_8);
  \draw [black,->,very thick] (node_8) ..controls (31.233bp,170.38bp) and (22.733bp,182.38bp)  .. (node_3);
  \draw [black,->] (node_9) ..controls (48.0bp,269.95bp) and (48.0bp,280.9bp)  .. (node_13);
  \draw [black,->] (node_2) ..controls (97.749bp,170.38bp) and (89.748bp,182.38bp)  .. (node_6);
  \draw [black,->] (node_0) ..controls (49.233bp,20.377bp) and (40.733bp,32.377bp)  .. (node_1);
  \draw [black,->] (node_6) ..controls (67.374bp,220.23bp) and (61.065bp,231.88bp)  .. (node_9);
  \draw [black,->] (node_0) ..controls (58.0bp,28.102bp) and (58.0bp,67.779bp)  .. (node_10);
  \draw [black,->,very thick] (node_10) ..controls (61.778bp,120.09bp) and (65.308bp,131.39bp)  .. (node_12);
\end{tikzpicture}
\end{center}

The induced quiver-with-relations on the marked vertices is just a
quiver, as one can check that no relation is implied by
the commuting relations in the initial Hasse diagram. Removing the vertex $2$ in
this quiver gives a quiver of affine type $\TA_{7}^{(1)}$.

Here is the similar wild sub-quiver for the other cambrian lattice of type $\TA_3$.

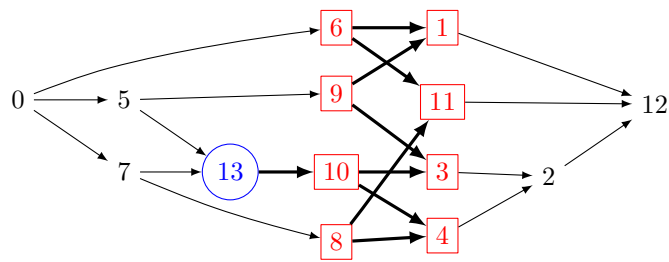
\begin{center}
\begin{tikzpicture}[rotate=270,>=latex,line join=bevel,scale=0.8]
  \node (node_0) at (40.0bp,7.0bp) [draw,draw=none] {$0$};
  \node (node_1) at (6.0bp,207.0bp) [draw,color=red] {$1$};
  \node (node_10) at (74.0bp,157.0bp) [draw,color=red] {$10$};
  \node (node_11) at (41.0bp,207.0bp) [draw,color=red] {$11$};
  \node (node_12) at (42.0bp,307.0bp) [draw,draw=none] {$12$};
  \node (node_13) at (74.0bp,107.0bp) [draw,color=blue,shape=circle] {$13$};
  \node (node_2) at (76.0bp,257.0bp) [draw,draw=none] {$2$};
  \node (node_3) at (74.0bp,207.0bp) [draw,color=red] {$3$};
  \node (node_4) at (104.0bp,207.0bp) [draw,color=red] {$4$};
  \node (node_5) at (40.0bp,57.0bp) [draw,draw=none] {$5$};
  \node (node_6) at (6.0bp,157.0bp) [draw,color=red] {$6$};
  \node (node_7) at (74.0bp,57.0bp) [draw,draw=none] {$7$};
  \node (node_8) at (107.0bp,157.0bp) [draw,color=red] {$8$};
  \node (node_9) at (37.0bp,157.0bp) [draw,color=red] {$9$};
  \draw [black,->,very thick] (node_10) ..controls (74.0bp,169.95bp) and (74.0bp,180.9bp)  .. (node_3);
  \draw [black,->,very thick] (node_8) ..controls (91.066bp,169.59bp) and (71.341bp,183.93bp)  .. (node_11);
  \draw [black,->,very thick] (node_6) ..controls (6.0bp,169.95bp) and (6.0bp,180.9bp)  .. (node_1);
  \draw [black,->] (node_2) ..controls (67.233bp,270.38bp) and (58.733bp,282.38bp)  .. (node_12);
  \draw [black,->] (node_7) ..controls (74.0bp,69.947bp) and (74.0bp,80.897bp)  .. (node_13);
  \draw [black,->] (node_3) ..controls (74.498bp,219.95bp) and (74.954bp,230.9bp)  .. (node_2);
  \draw [black,->,very thick] (node_9) ..controls (46.595bp,170.45bp) and (55.984bp,182.63bp)  .. (node_3);
  \draw [black,->] (node_11) ..controls (41.205bp,228.1bp) and (41.61bp,267.78bp)  .. (node_12);
  \draw [black,->] (node_5) ..controls (39.385bp,78.102bp) and (38.17bp,117.78bp)  .. (node_9);
  \draw [black,->,very thick] (node_8) ..controls (106.25bp,169.95bp) and (105.57bp,180.9bp)  .. (node_4);
  \draw [black,->] (node_5) ..controls (48.767bp,70.377bp) and (57.267bp,82.377bp)  .. (node_13);
  \draw [black,->] (node_1) ..controls (13.434bp,228.24bp) and (28.229bp,268.51bp)  .. (node_12);
  \draw [black,->,very thick] (node_10) ..controls (81.691bp,170.31bp) and (89.08bp,182.13bp)  .. (node_4);
  \draw [black,->] (node_0) ..controls (48.767bp,20.377bp) and (57.267bp,32.377bp)  .. (node_7);
  \draw [black,->] (node_0) ..controls (40.0bp,19.947bp) and (40.0bp,30.897bp)  .. (node_5);
  \draw [black,->] (node_7) ..controls (80.005bp,71.627bp) and (86.443bp,86.76bp)  .. (91.0bp,100.0bp) .. controls (95.64bp,113.48bp) and (100.02bp,129.1bp)  .. (node_8);
  \draw [black,->] (node_0) ..controls (34.413bp,21.579bp) and (28.546bp,36.682bp)  .. (25.0bp,50.0bp) .. controls (16.598bp,81.56bp) and (10.815bp,119.41bp)  .. (node_6);
  \draw [black,->,very thick] (node_6) ..controls (15.077bp,170.45bp) and (23.958bp,182.63bp)  .. (node_11);
  \draw [black,->,very thick] (node_9) ..controls (29.053bp,170.31bp) and (21.417bp,182.13bp)  .. (node_1);
  \draw [black,->,very thick] (node_13) ..controls (74.0bp,119.95bp) and (74.0bp,130.9bp)  .. (node_10);
  \draw [black,->] (node_4) ..controls (96.822bp,220.31bp) and (89.925bp,232.13bp)  .. (node_2);
\end{tikzpicture}
\end{center}

\subsection{Type $\TB_3$}

Let us now consider one of the cambrian lattices of type $\TB_3$, with
$20$ vertices. Let us again show that this has wild representation
type by exhibiting a wild subquiver.

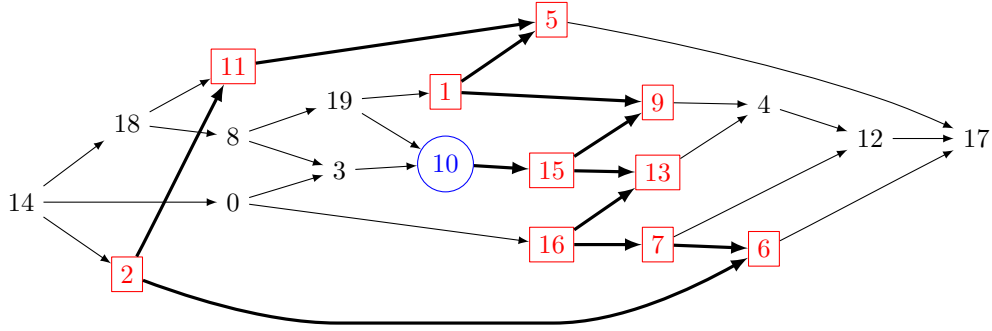
\begin{center}
\begin{tikzpicture}[rotate=270,>=latex,line join=bevel,scale=0.8]
\node (node_13) at (78.0bp,307.0bp) [draw,color=red] {$13$};
  \node (node_14) at (92.0bp,7.0bp) [draw,draw=none] {$14$};
  \node (node_18) at (55.0bp,57.0bp) [draw,draw=none] {$18$};
  \node (node_19) at (44.0bp,157.0bp) [draw,draw=none] {$19$};
  \node (node_9) at (45.0bp,307.0bp) [draw,color=red] {$9$};
  \node (node_8) at (61.0bp,107.0bp) [draw,draw=none] {$8$};
  \node (node_17) at (62.0bp,457.0bp) [draw,draw=none] {$17$};
  \node (node_16) at (112.0bp,257.0bp) [draw,color=red] {$16$};
  \node (node_5) at (6.0bp,257.0bp) [draw,color=red] {$5$};
  \node (node_4) at (46.0bp,357.0bp) [draw,draw=none] {$4$};
  \node (node_3) at (77.0bp,157.0bp) [draw,draw=none] {$3$};
  \node (node_2) at (126.0bp,57.0bp) [draw,color=red] {$2$};
  \node (node_1) at (40.0bp,207.0bp) [draw,color=red] {$1$};
  \node (node_0) at (92.0bp,107.0bp) [draw,draw=none] {$0$};
  \node (node_7) at (112.0bp,307.0bp) [draw,color=red] {$7$};
  \node (node_11) at (28.0bp,107.0bp) [draw,color=red] {$11$};
  \node (node_6) at (114.0bp,357.0bp) [draw,color=red] {$6$};
  \node (node_10) at (74.0bp,207.0bp) [draw,color=blue,shape=circle] {$10$};
  \node (node_15) at (77.0bp,257.0bp) [draw,color=red] {$15$};
  \node (node_12) at (62.0bp,407.0bp) [draw,draw=none] {$12$};
  \draw [black,->] (node_3) ..controls (76.253bp,169.95bp) and (75.569bp,180.9bp)  .. (node_10);
  \draw [black,->,very thick] (node_2) ..controls (105.19bp,68.193bp) and (67.584bp,86.612bp)  .. (node_11);
  \draw [black,->] (node_8) ..controls (56.718bp,120.09bp) and (52.717bp,131.39bp)  .. (node_19);
  \draw [black,->] (node_18) ..controls (48.119bp,70.233bp) and (41.568bp,81.88bp)  .. (node_11);
  \draw [black,->,very thick] (node_10) ..controls (74.747bp,219.95bp) and (75.431bp,230.9bp)  .. (node_15);
  \draw [black,->] (node_18) ..controls (56.493bp,69.947bp) and (57.862bp,80.897bp)  .. (node_8);
  \draw [black,->,very thick] (node_7) ..controls (112.5bp,319.95bp) and (112.95bp,330.9bp)  .. (node_6);
  \draw [black,->] (node_12) ..controls (62.0bp,419.95bp) and (62.0bp,430.9bp)  .. (node_17);
  \draw [black,->] (node_14) ..controls (92.0bp,28.102bp) and (92.0bp,67.779bp)  .. (node_0);
  \draw [black,->] (node_0) ..controls (88.222bp,120.09bp) and (84.692bp,131.39bp)  .. (node_3);
  \draw [black,->] (node_7) ..controls (101.61bp,328.37bp) and (80.75bp,369.25bp)  .. (node_12);
  \draw [black,->] (node_6) ..controls (103.19bp,378.37bp) and (81.5bp,419.25bp)  .. (node_17);
  \draw [black,->,very thick] (node_15) ..controls (68.749bp,270.38bp) and (60.748bp,282.38bp)  .. (node_9);
  \draw [black,->,very thick] (node_15) ..controls (77.249bp,269.95bp) and (77.477bp,280.9bp)  .. (node_13);
  \draw [black,->] (node_8) ..controls (65.03bp,120.09bp) and (68.795bp,131.39bp)  .. (node_3);
  \draw [black,->] (node_4) ..controls (50.03bp,370.09bp) and (53.795bp,381.39bp)  .. (node_12);
  \draw [black,->] (node_19) ..controls (43.004bp,169.95bp) and (42.092bp,180.9bp)  .. (node_1);
  \draw [black,->] (node_5) ..controls (10.73bp,284.47bp) and (24.42bp,356.77bp)  .. (44.0bp,414.0bp) .. controls (47.142bp,423.19bp) and (51.394bp,433.19bp)  .. (node_17);
  \draw [black,->,very thick] (node_1) ..controls (31.233bp,220.38bp) and (22.733bp,232.38bp)  .. (node_5);
  \draw [black,->] (node_9) ..controls (45.249bp,319.95bp) and (45.477bp,330.9bp)  .. (node_4);
  \draw [black,->,very thick] (node_1) ..controls (41.026bp,228.1bp) and (43.05bp,267.78bp)  .. (node_9);
  \draw [black,->] (node_0) ..controls (95.569bp,134.41bp) and (105.43bp,207.41bp)  .. (node_16);
  \draw [black,->] (node_14) ..controls (100.77bp,20.377bp) and (109.27bp,32.377bp)  .. (node_2);
  \draw [black,->] (node_14) ..controls (82.405bp,20.448bp) and (73.016bp,32.628bp)  .. (node_18);
  \draw [black,->,very thick] (node_11) ..controls (24.061bp,134.5bp) and (13.127bp,208.06bp)  .. (node_5);
  \draw [black,->] (node_13) ..controls (69.749bp,320.38bp) and (61.748bp,332.38bp)  .. (node_4);
  \draw [black,->,very thick] (node_2) ..controls (134.11bp,78.48bp) and (149.0bp,120.0bp)  .. (149.0bp,156.0bp) .. controls (149.0bp,156.0bp) and (149.0bp,156.0bp)  .. (149.0bp,258.0bp) .. controls (149.0bp,289.1bp) and (133.18bp,322.68bp)  .. (node_6);
  \draw [black,->] (node_19) ..controls (51.691bp,170.31bp) and (59.08bp,182.13bp)  .. (node_10);
  \draw [black,->,very thick] (node_16) ..controls (112.0bp,269.95bp) and (112.0bp,280.9bp)  .. (node_7);
  \draw [black,->,very thick] (node_16) ..controls (103.23bp,270.38bp) and (94.733bp,282.38bp)  .. (node_13);
\end{tikzpicture}
\end{center}

The induced quiver-with-relations on the marked vertices is just a
quiver, as one can check that no relation is implied by
the relations in the initial Hasse diagram. Removing the vertex $10$ in
this quiver gives a quiver of affine type $\TA_{9}^{(1)}$.

Here is the similar wild sub-quiver for the other cambrian lattice of type $\TB_3$.

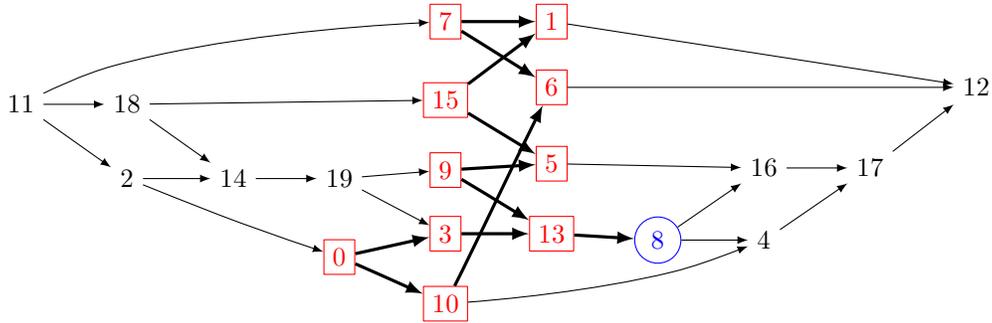
\begin{center}
\begin{tikzpicture}[rotate=270,>=latex,line join=bevel,scale=0.8]
  \node (node_0) at (117.0bp,157.0bp) [draw,color=red] {$0$};
  \node (node_1) at (6.0bp,257.0bp) [draw,color=red] {$1$};
  \node (node_10) at (139.0bp,207.0bp) [draw,color=red] {$10$};
  \node (node_11) at (45.0bp,7.0bp) [draw,draw=none] {$11$};
  \node (node_12) at (37.0bp,457.0bp) [draw,draw=none] {$12$};
  \node (node_13) at (106.0bp,257.0bp) [draw,color=red] {$13$};
  \node (node_14) at (80.0bp,107.0bp) [draw,draw=none] {$14$};
  \node (node_15) at (43.0bp,207.0bp) [draw,color=red] {$15$};
  \node (node_16) at (75.0bp,357.0bp) [draw,draw=none] {$16$};
  \node (node_17) at (75.0bp,407.0bp) [draw,draw=none] {$17$};
  \node (node_18) at (45.0bp,57.0bp) [draw,draw=none] {$18$};
  \node (node_19) at (80.0bp,157.0bp) [draw,draw=none] {$19$};
  \node (node_2) at (80.0bp,57.0bp) [draw,draw=none] {$2$};
  \node (node_3) at (106.0bp,207.0bp) [draw,color=red] {$3$};
  \node (node_4) at (109.0bp,357.0bp) [draw,draw=none] {$4$};
  \node (node_5) at (73.0bp,257.0bp) [draw,color=red] {$5$};
  \node (node_6) at (37.0bp,257.0bp) [draw,color=red] {$6$};
  \node (node_7) at (6.0bp,207.0bp) [draw,color=red] {$7$};
  \node (node_8) at (109.0bp,307.0bp) [draw,color=blue,shape=circle] {$8$};
  \node (node_9) at (76.0bp,207.0bp) [draw,color=red] {$9$};
  \draw [black,->] (node_6) ..controls (37.0bp,289.91bp) and (37.0bp,399.07bp)  .. (node_12);
  \draw [black,->,very thick] (node_9) ..controls (75.253bp,219.95bp) and (74.569bp,230.9bp)  .. (node_5);
  \draw [black,->] (node_17) ..controls (65.145bp,420.45bp) and (55.503bp,432.63bp)  .. (node_12);
  \draw [black,->,very thick] (node_15) ..controls (50.691bp,220.31bp) and (58.08bp,232.13bp)  .. (node_5);
  \draw [black,->] (node_4) ..controls (100.23bp,370.38bp) and (91.733bp,382.38bp)  .. (node_17);
  \draw [black,->] (node_16) ..controls (75.0bp,369.95bp) and (75.0bp,380.9bp)  .. (node_17);
  \draw [black,->,very thick] (node_7) ..controls (13.947bp,220.31bp) and (21.583bp,232.13bp)  .. (node_6);
  \draw [black,->,very thick] (node_15) ..controls (33.405bp,220.45bp) and (24.016bp,232.63bp)  .. (node_1);
  \draw [black,->,very thick] (node_10) ..controls (113.48bp,220.01bp) and (73.977bp,238.6bp)  .. (node_6);
  \draw [black,->] (node_8) ..controls (109.0bp,319.95bp) and (109.0bp,330.9bp)  .. (node_4);
  \draw [black,->] (node_11) ..controls (37.992bp,21.4bp) and (30.758bp,36.388bp)  .. (27.0bp,50.0bp) .. controls (13.192bp,100.02bp) and (8.3144bp,161.66bp)  .. (node_7);
  \draw [black,->,very thick] (node_0) ..controls (114.25bp,170.02bp) and (111.7bp,181.14bp)  .. (node_3);
  \draw [black,->,very thick] (node_7) ..controls (6.0bp,219.95bp) and (6.0bp,230.9bp)  .. (node_1);
  \draw [black,->] (node_19) ..controls (79.004bp,169.95bp) and (78.092bp,180.9bp)  .. (node_9);
  \draw [black,->,very thick] (node_3) ..controls (106.0bp,219.95bp) and (106.0bp,230.9bp)  .. (node_13);
  \draw [black,->] (node_11) ..controls (45.0bp,19.947bp) and (45.0bp,30.897bp)  .. (node_18);
  \draw [black,->] (node_11) ..controls (54.077bp,20.448bp) and (62.958bp,32.628bp)  .. (node_2);
  \draw [black,->,very thick] (node_0) ..controls (122.57bp,170.16bp) and (127.83bp,181.63bp)  .. (node_10);
  \draw [black,->,very thick] (node_9) ..controls (83.691bp,220.31bp) and (91.08bp,232.13bp)  .. (node_13);
  \draw [black,->] (node_8) ..controls (100.23bp,320.38bp) and (91.733bp,332.38bp)  .. (node_16);
  \draw [black,->] (node_2) ..controls (85.745bp,71.721bp) and (92.017bp,86.914bp)  .. (97.0bp,100.0bp) .. controls (102.23bp,113.75bp) and (107.84bp,129.54bp)  .. (node_0);
  \draw [black,->] (node_10) ..controls (137.37bp,229.95bp) and (133.26bp,276.47bp)  .. (124.0bp,314.0bp) .. controls (121.76bp,323.09bp) and (118.28bp,332.97bp)  .. (node_4);
  \draw [black,->] (node_18) ..controls (54.077bp,70.448bp) and (62.958bp,82.628bp)  .. (node_14);
  \draw [black,->] (node_5) ..controls (73.41bp,278.1bp) and (74.22bp,317.78bp)  .. (node_16);
  \draw [black,->] (node_2) ..controls (80.0bp,69.947bp) and (80.0bp,80.897bp)  .. (node_14);
  \draw [black,->,very thick] (node_13) ..controls (106.75bp,269.95bp) and (107.43bp,280.9bp)  .. (node_8);
  \draw [black,->] (node_1) ..controls (10.996bp,289.91bp) and (28.086bp,399.07bp)  .. (node_12);
  \draw [black,->] (node_18) ..controls (44.643bp,84.409bp) and (43.657bp,157.41bp)  .. (node_15);
  \draw [black,->] (node_19) ..controls (86.626bp,170.23bp) and (92.935bp,181.88bp)  .. (node_3);
  \draw [black,->] (node_14) ..controls (80.0bp,119.95bp) and (80.0bp,130.9bp)  .. (node_19);
\end{tikzpicture}
\end{center}

\subsection{Type $\TyH_3$}

Let us now consider one of the cambrian lattices of type $\TyH_3$, with
$32$ vertices. Let us again show that this has wild representation
type by exhibiting a wild subquiver.

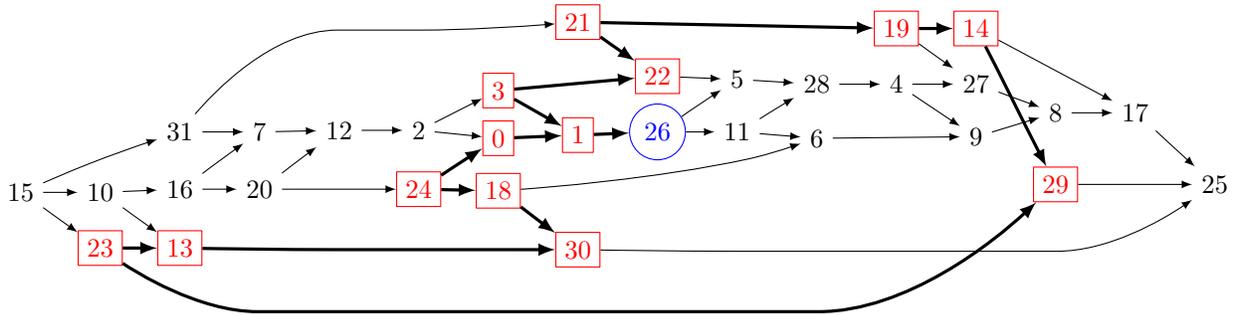
\begin{center}
\begin{tikzpicture}[rotate=270,>=latex,line join=bevel,scale=0.6]
  \node (node_0) at (81.0bp,307.0bp) [draw,color=red] {$0$};
  \node (node_1) at (79.0bp,357.0bp) [draw,color=red] {$1$};
  \node (node_10) at (115.0bp,57.0bp) [draw,draw=none] {$10$};
  \node (node_11) at (77.0bp,457.0bp) [draw,draw=none] {$11$};
  \node (node_12) at (76.0bp,207.0bp) [draw,draw=none] {$12$};
  \node (node_13) at (150.0bp,107.0bp) [draw,color=red] {$13$};
  \node (node_14) at (12.0bp,607.0bp) [draw,color=red] {$14$};
  \node (node_15) at (115.0bp,7.0bp) [draw,draw=none] {$15$};
  \node (node_16) at (113.0bp,107.0bp) [draw,draw=none] {$16$};
  \node (node_17) at (65.0bp,707.0bp) [draw,draw=none] {$17$};
  \node (node_18) at (114.0bp,307.0bp) [draw,color=red] {$18$};
  \node (node_19) at (12.0bp,557.0bp) [draw,color=red] {$19$};
  \node (node_2) at (76.0bp,257.0bp) [draw,draw=none] {$2$};
  \node (node_20) at (113.0bp,157.0bp) [draw,draw=none] {$20$};
  \node (node_21) at (8.0bp,357.0bp) [draw,color=red] {$21$};
  \node (node_22) at (42.0bp,407.0bp) [draw,color=red] {$22$};
  \node (node_23) at (150.0bp,57.0bp) [draw,color=red] {$23$};
  \node (node_24) at (113.0bp,257.0bp) [draw,color=red] {$24$};
  \node (node_25) at (110.0bp,757.0bp) [draw,draw=none] {$25$};
  \node (node_26) at (77.0bp,407.0bp) [draw,color=blue,shape=circle] {$26$};
  \node (node_27) at (47.0bp,607.0bp) [draw,draw=none] {$27$};
  \node (node_28) at (47.0bp,507.0bp) [draw,draw=none] {$28$};
  \node (node_29) at (110.0bp,657.0bp) [draw,color=red] {$29$};
  \node (node_3) at (51.0bp,307.0bp) [draw,color=red] {$3$};
  \node (node_30) at (151.0bp,357.0bp) [draw,color=red] {$30$};
  \node (node_31) at (77.0bp,107.0bp) [draw,draw=none] {$31$};
  \node (node_4) at (47.0bp,557.0bp) [draw,draw=none] {$4$};
  \node (node_5) at (44.0bp,457.0bp) [draw,draw=none] {$5$};
  \node (node_6) at (81.0bp,507.0bp) [draw,draw=none] {$6$};
  \node (node_7) at (77.0bp,157.0bp) [draw,draw=none] {$7$};
  \node (node_8) at (65.0bp,657.0bp) [draw,draw=none] {$8$};
  \node (node_9) at (80.0bp,607.0bp) [draw,draw=none] {$9$};
  \draw [black,->] (node_29) ..controls (110.0bp,678.1bp) and (110.0bp,717.78bp)  .. (node_25);
  \draw [black,->,very thick] (node_19) ..controls (12.0bp,569.95bp) and (12.0bp,580.9bp)  .. (node_14);
  \draw [black,->,very thick] (node_14) ..controls (36.267bp,619.89bp) and (71.221bp,637.01bp)  .. (node_29);
  \draw [black,->] (node_15) ..controls (115.0bp,19.947bp) and (115.0bp,30.897bp)  .. (node_10);
  \draw [black,->] (node_22) ..controls (42.498bp,419.95bp) and (42.954bp,430.9bp)  .. (node_5);
  \draw [black,->] (node_31) ..controls (77.0bp,119.95bp) and (77.0bp,130.9bp)  .. (node_7);
  \draw [black,->,very thick] (node_21) ..controls (8.6447bp,389.91bp) and (10.85bp,499.07bp)  .. (node_19);
  \draw [black,->,very thick] (node_13) ..controls (150.35bp,128.85bp) and (151.0bp,170.95bp)  .. (151.0bp,206.0bp) .. controls (151.0bp,206.0bp) and (151.0bp,206.0bp)  .. (151.0bp,258.0bp) .. controls (151.0bp,287.02bp) and (151.0bp,320.88bp)  .. (node_30);
  \draw [black,->,very thick] (node_23) ..controls (150.0bp,69.947bp) and (150.0bp,80.897bp)  .. (node_13);
  \draw [black,->,very thick] (node_18) ..controls (123.6bp,320.45bp) and (132.98bp,332.63bp)  .. (node_30);
  \draw [black,->] (node_15) ..controls (108.85bp,21.684bp) and (102.16bp,36.855bp)  .. (97.0bp,50.0bp) .. controls (91.626bp,63.695bp) and (86.038bp,79.495bp)  .. (node_31);
  \draw [black,->] (node_2) ..controls (69.628bp,270.23bp) and (63.563bp,281.88bp)  .. (node_3);
  \draw [black,->] (node_4) ..controls (47.0bp,569.95bp) and (47.0bp,580.9bp)  .. (node_27);
  \draw [black,->] (node_7) ..controls (76.751bp,169.95bp) and (76.523bp,180.9bp)  .. (node_12);
  \draw [black,->,very thick] (node_1) ..controls (78.502bp,369.95bp) and (78.046bp,380.9bp)  .. (node_26);
  \draw [black,->] (node_6) ..controls (80.795bp,528.1bp) and (80.39bp,567.78bp)  .. (node_9);
  \draw [black,->] (node_16) ..controls (103.66bp,120.45bp) and (94.529bp,132.63bp)  .. (node_7);
  \draw [black,->] (node_17) ..controls (76.804bp,720.59bp) and (88.562bp,733.13bp)  .. (node_25);
  \draw [black,->] (node_28) ..controls (47.0bp,519.95bp) and (47.0bp,530.9bp)  .. (node_4);
  \draw [black,->,very thick] (node_23) ..controls (164.6bp,78.539bp) and (190.0bp,118.51bp)  .. (190.0bp,156.0bp) .. controls (190.0bp,156.0bp) and (190.0bp,156.0bp)  .. (190.0bp,508.0bp) .. controls (190.0bp,563.91bp) and (145.27bp,618.96bp)  .. (node_29);
  \draw [black,->] (node_10) ..controls (114.5bp,69.947bp) and (114.05bp,80.897bp)  .. (node_16);
  \draw [black,->] (node_26) ..controls (77.0bp,419.95bp) and (77.0bp,430.9bp)  .. (node_11);
  \draw [black,->] (node_11) ..controls (77.996bp,469.95bp) and (78.908bp,480.9bp)  .. (node_6);
  \draw [black,->,very thick] (node_3) ..controls (58.178bp,320.31bp) and (65.075bp,332.13bp)  .. (node_1);
  \draw [black,->] (node_27) ..controls (51.534bp,620.09bp) and (55.77bp,631.39bp)  .. (node_8);
  \draw [black,->,very thick] (node_24) ..controls (104.75bp,270.38bp) and (96.748bp,282.38bp)  .. (node_0);
  \draw [black,->] (node_2) ..controls (77.245bp,269.95bp) and (78.385bp,280.9bp)  .. (node_0);
  \draw [black,->,very thick] (node_21) ..controls (16.767bp,370.38bp) and (25.267bp,382.38bp)  .. (node_22);
  \draw [black,->] (node_5) ..controls (44.747bp,469.95bp) and (45.431bp,480.9bp)  .. (node_28);
  \draw [black,->] (node_31) ..controls (52.848bp,127.67bp) and (13.0bp,164.95bp)  .. (13.0bp,206.0bp) .. controls (13.0bp,206.0bp) and (13.0bp,206.0bp)  .. (13.0bp,258.0bp) .. controls (13.0bp,287.06bp) and (10.78bp,320.91bp)  .. (node_21);
  \draw [black,->] (node_9) ..controls (76.222bp,620.09bp) and (72.692bp,631.39bp)  .. (node_8);
  \draw [black,->,very thick] (node_0) ..controls (80.502bp,319.95bp) and (80.046bp,330.9bp)  .. (node_1);
  \draw [black,->] (node_30) ..controls (151.35bp,378.85bp) and (152.0bp,420.95bp)  .. (152.0bp,456.0bp) .. controls (152.0bp,456.0bp) and (152.0bp,456.0bp)  .. (152.0bp,658.0bp) .. controls (152.0bp,689.86bp) and (133.01bp,723.18bp)  .. (node_25);
  \draw [black,->] (node_8) ..controls (65.0bp,669.95bp) and (65.0bp,680.9bp)  .. (node_17);
  \draw [black,->] (node_12) ..controls (76.0bp,219.95bp) and (76.0bp,230.9bp)  .. (node_2);
  \draw [black,->] (node_10) ..controls (124.08bp,70.448bp) and (132.96bp,82.628bp)  .. (node_13);
  \draw [black,->,very thick] (node_24) ..controls (113.25bp,269.95bp) and (113.48bp,280.9bp)  .. (node_18);
  \draw [black,->,very thick] (node_3) ..controls (49.154bp,328.1bp) and (45.51bp,367.78bp)  .. (node_22);
  \draw [black,->] (node_26) ..controls (68.491bp,420.38bp) and (60.241bp,432.38bp)  .. (node_5);
  \draw [black,->] (node_20) ..controls (103.4bp,170.45bp) and (94.016bp,182.63bp)  .. (node_12);
  \draw [black,->] (node_11) ..controls (69.309bp,470.31bp) and (61.92bp,482.13bp)  .. (node_28);
  \draw [black,->] (node_14) ..controls (23.017bp,628.37bp) and (45.125bp,669.25bp)  .. (node_17);
  \draw [black,->] (node_16) ..controls (113.0bp,119.95bp) and (113.0bp,130.9bp)  .. (node_20);
  \draw [black,->] (node_15) ..controls (124.08bp,20.448bp) and (132.96bp,32.628bp)  .. (node_23);
  \draw [black,->] (node_20) ..controls (113.0bp,178.1bp) and (113.0bp,217.78bp)  .. (node_24);
  \draw [black,->] (node_19) ..controls (21.077bp,570.45bp) and (29.958bp,582.63bp)  .. (node_27);
  \draw [black,->] (node_4) ..controls (55.509bp,570.38bp) and (63.759bp,582.38bp)  .. (node_9);
  \draw [black,->] (node_18) ..controls (112.2bp,334.27bp) and (106.57bp,406.1bp)  .. (94.0bp,464.0bp) .. controls (92.04bp,473.02bp) and (89.029bp,482.9bp)  .. (node_6);
\end{tikzpicture}
\end{center}

The induced quiver-with-relations on the marked vertices is just a
quiver, as one can check that no relation is implied by
the relations in the initial Hasse diagram. Removing the vertex $26$ in
this quiver gives a quiver of affine type $\TA_{12}^{(1)}$.

Here is the similar wild sub-quiver for the other cambrian lattice of type $\TyH_3$.

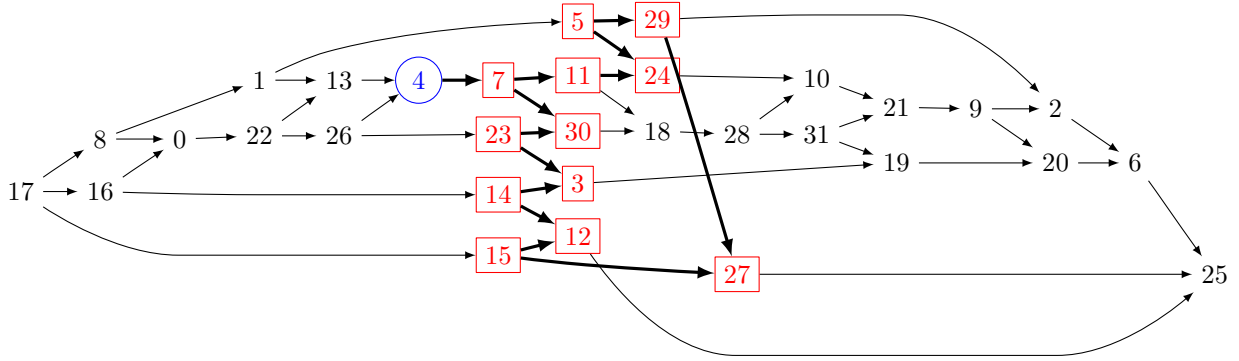
\begin{center}
\begin{tikzpicture}[rotate=270,>=latex,line join=bevel,scale=0.6]
  \node (node_0) at (83.0bp,107.0bp) [draw,draw=none] {$0$};
  \node (node_1) at (46.0bp,157.0bp) [draw,draw=none] {$1$};
  \node (node_10) at (45.0bp,507.0bp) [draw,draw=none] {$10$};
  \node (node_11) at (43.0bp,357.0bp) [draw,color=red] {$11$};
  \node (node_12) at (144.0bp,357.0bp) [draw,color=red] {$12$};
  \node (node_13) at (46.0bp,207.0bp) [draw,draw=none] {$13$};
  \node (node_14) at (118.0bp,307.0bp) [draw,color=red] {$14$};
  \node (node_15) at (156.0bp,307.0bp) [draw,color=red] {$15$};
  \node (node_16) at (116.0bp,57.0bp) [draw,draw=none] {$16$};
  \node (node_17) at (116.0bp,7.0bp) [draw,draw=none] {$17$};
  \node (node_18) at (78.0bp,407.0bp) [draw,draw=none] {$18$};
  \node (node_19) at (98.0bp,557.0bp) [draw,draw=none] {$19$};
  \node (node_2) at (64.0bp,657.0bp) [draw,draw=none] {$2$};
  \node (node_20) at (98.0bp,657.0bp) [draw,draw=none] {$20$};
  \node (node_21) at (63.0bp,557.0bp) [draw,draw=none] {$21$};
  \node (node_22) at (81.0bp,157.0bp) [draw,draw=none] {$22$};
  \node (node_23) at (80.0bp,307.0bp) [draw,color=red] {$23$};
  \node (node_24) at (43.0bp,407.0bp) [draw,color=red] {$24$};
  \node (node_25) at (168.0bp,757.0bp) [draw,draw=none] {$25$};
  \node (node_26) at (81.0bp,207.0bp) [draw,draw=none] {$26$};
  \node (node_27) at (168.0bp,457.0bp) [draw,color=red] {$27$};
  \node (node_28) at (80.0bp,457.0bp) [draw,draw=none] {$28$};
  \node (node_29) at (8.0bp,407.0bp) [draw,color=red] {$29$};
  \node (node_3) at (111.0bp,357.0bp) [draw,color=red] {$3$};
  \node (node_30) at (78.0bp,357.0bp) [draw,color=red] {$30$};
  \node (node_31) at (80.0bp,507.0bp) [draw,draw=none] {$31$};
  \node (node_4) at (46.0bp,257.0bp) [draw,color=blue,shape=circle] {$4$};
  \node (node_5) at (9.0bp,357.0bp) [draw,color=red] {$5$};
  \node (node_6) at (98.0bp,707.0bp) [draw,draw=none] {$6$};
  \node (node_7) at (46.0bp,307.0bp) [draw,color=red] {$7$};
  \node (node_8) at (83.0bp,57.0bp) [draw,draw=none] {$8$};
  \node (node_9) at (64.0bp,607.0bp) [draw,draw=none] {$9$};
  \draw [black,->,very thick] (node_5) ..controls (8.7511bp,369.95bp) and (8.523bp,380.9bp)  .. (node_29);
  \draw [black,->] (node_26) ..controls (71.923bp,220.45bp) and (63.042bp,232.63bp)  .. (node_4);
  \draw [black,->] (node_22) ..controls (71.923bp,170.45bp) and (63.042bp,182.63bp)  .. (node_13);
  \draw [black,->] (node_28) ..controls (70.923bp,470.45bp) and (62.042bp,482.63bp)  .. (node_10);
  \draw [black,->] (node_10) ..controls (49.534bp,520.09bp) and (53.77bp,531.39bp)  .. (node_21);
  \draw [black,->,very thick] (node_11) ..controls (43.0bp,369.95bp) and (43.0bp,380.9bp)  .. (node_24);
  \draw [black,->] (node_28) ..controls (80.0bp,469.95bp) and (80.0bp,480.9bp)  .. (node_31);
  \draw [black,->] (node_17) ..controls (116.0bp,19.947bp) and (116.0bp,30.897bp)  .. (node_16);
  \draw [black,->] (node_27) ..controls (168.0bp,478.85bp) and (168.0bp,520.95bp)  .. (168.0bp,556.0bp) .. controls (168.0bp,556.0bp) and (168.0bp,556.0bp)  .. (168.0bp,658.0bp) .. controls (168.0bp,687.02bp) and (168.0bp,720.88bp)  .. (node_25);
  \draw [black,->] (node_26) ..controls (80.795bp,228.1bp) and (80.39bp,267.78bp)  .. (node_23);
  \draw [black,->] (node_19) ..controls (98.0bp,578.1bp) and (98.0bp,617.78bp)  .. (node_20);
  \draw [black,->] (node_8) ..controls (83.0bp,69.947bp) and (83.0bp,80.897bp)  .. (node_0);
  \draw [black,->] (node_24) ..controls (43.41bp,428.1bp) and (44.22bp,467.78bp)  .. (node_10);
  \draw [black,->] (node_11) ..controls (52.077bp,370.45bp) and (60.958bp,382.63bp)  .. (node_18);
  \draw [black,->] (node_17) ..controls (130.6bp,28.539bp) and (156.0bp,68.512bp)  .. (156.0bp,106.0bp) .. controls (156.0bp,106.0bp) and (156.0bp,106.0bp)  .. (156.0bp,208.0bp) .. controls (156.0bp,237.02bp) and (156.0bp,270.88bp)  .. (node_15);
  \draw [black,->,very thick] (node_7) ..controls (54.251bp,320.38bp) and (62.252bp,332.38bp)  .. (node_30);
  \draw [black,->] (node_16) ..controls (116.7bp,78.849bp) and (118.0bp,120.94bp)  .. (118.0bp,156.0bp) .. controls (118.0bp,156.0bp) and (118.0bp,156.0bp)  .. (118.0bp,208.0bp) .. controls (118.0bp,237.02bp) and (118.0bp,270.88bp)  .. (node_14);
  \draw [black,->] (node_12) ..controls (171.84bp,376.37bp) and (219.0bp,412.51bp)  .. (219.0bp,456.0bp) .. controls (219.0bp,456.0bp) and (219.0bp,456.0bp)  .. (219.0bp,658.0bp) .. controls (219.0bp,691.32bp) and (195.52bp,724.55bp)  .. (node_25);
  \draw [black,->] (node_31) ..controls (84.534bp,520.09bp) and (88.77bp,531.39bp)  .. (node_19);
  \draw [black,->] (node_6) ..controls (115.0bp,719.66bp) and (136.22bp,734.21bp)  .. (node_25);
  \draw [black,->,very thick] (node_14) ..controls (124.63bp,320.23bp) and (130.93bp,331.88bp)  .. (node_12);
  \draw [black,->,very thick] (node_7) ..controls (45.253bp,319.95bp) and (44.569bp,330.9bp)  .. (node_11);
  \draw [black,->] (node_8) ..controls (75.36bp,78.236bp) and (60.153bp,118.51bp)  .. (node_1);
  \draw [black,->] (node_1) ..controls (38.94bp,171.39bp) and (31.672bp,186.36bp)  .. (28.0bp,200.0bp) .. controls (14.53bp,250.02bp) and (10.589bp,311.66bp)  .. (node_5);
  \draw [black,->] (node_9) ..controls (64.0bp,619.95bp) and (64.0bp,630.9bp)  .. (node_2);
  \draw [black,->] (node_16) ..controls (107.49bp,70.377bp) and (99.241bp,82.377bp)  .. (node_0);
  \draw [black,->] (node_0) ..controls (82.502bp,119.95bp) and (82.046bp,130.9bp)  .. (node_22);
  \draw [black,->,very thick] (node_15) ..controls (153.0bp,320.02bp) and (150.21bp,331.14bp)  .. (node_12);
  \draw [black,->] (node_3) ..controls (108.9bp,389.91bp) and (101.74bp,499.07bp)  .. (node_19);
  \draw [black,->,very thick] (node_29) ..controls (19.035bp,412.0bp) and (22.141bp,413.08bp)  .. (25.0bp,414.0bp) .. controls (69.579bp,428.4bp) and (122.63bp,443.44bp)  .. (node_27);
  \draw [black,->] (node_1) ..controls (46.0bp,169.95bp) and (46.0bp,180.9bp)  .. (node_13);
  \draw [black,->] (node_18) ..controls (78.498bp,419.95bp) and (78.954bp,430.9bp)  .. (node_28);
  \draw [black,->] (node_29) ..controls (6.9426bp,428.85bp) and (5.0bp,470.93bp)  .. (5.0bp,506.0bp) .. controls (5.0bp,506.0bp) and (5.0bp,506.0bp)  .. (5.0bp,558.0bp) .. controls (5.0bp,592.82bp) and (32.669bp,625.91bp)  .. (node_2);
  \draw [black,->] (node_20) ..controls (98.0bp,669.95bp) and (98.0bp,680.9bp)  .. (node_6);
  \draw [black,->,very thick] (node_23) ..controls (87.947bp,320.31bp) and (95.583bp,332.13bp)  .. (node_3);
  \draw [black,->] (node_31) ..controls (75.718bp,520.09bp) and (71.717bp,531.39bp)  .. (node_21);
  \draw [black,->,very thick] (node_15) ..controls (157.85bp,321.76bp) and (159.8bp,336.98bp)  .. (161.0bp,350.0bp) .. controls (163.95bp,381.93bp) and (166.12bp,419.39bp)  .. (node_27);
  \draw [black,->,very thick] (node_4) ..controls (46.0bp,269.95bp) and (46.0bp,280.9bp)  .. (node_7);
  \draw [black,->] (node_17) ..controls (107.49bp,20.377bp) and (99.241bp,32.377bp)  .. (node_8);
  \draw [black,->,very thick] (node_5) ..controls (17.767bp,370.38bp) and (26.267bp,382.38bp)  .. (node_24);
  \draw [black,->,very thick] (node_23) ..controls (79.502bp,319.95bp) and (79.046bp,330.9bp)  .. (node_30);
  \draw [black,->] (node_13) ..controls (46.0bp,219.95bp) and (46.0bp,230.9bp)  .. (node_4);
  \draw [black,->] (node_9) ..controls (72.767bp,620.38bp) and (81.267bp,632.38bp)  .. (node_20);
  \draw [black,->] (node_2) ..controls (72.767bp,670.38bp) and (81.267bp,682.38bp)  .. (node_6);
  \draw [black,->] (node_21) ..controls (63.249bp,569.95bp) and (63.477bp,580.9bp)  .. (node_9);
  \draw [black,->] (node_30) ..controls (78.0bp,369.95bp) and (78.0bp,380.9bp)  .. (node_18);
  \draw [black,->,very thick] (node_14) ..controls (116.26bp,319.95bp) and (114.66bp,330.9bp)  .. (node_3);
  \draw [black,->] (node_22) ..controls (81.0bp,169.95bp) and (81.0bp,180.9bp)  .. (node_26);
\end{tikzpicture}
\end{center}

\begin{remark}
  Note the similarity in the shape of the wild subquivers for $\TA_3$,
  $\TB_3$ and $\TyH_3$. They all consist of something like a middle
  belt between a left part and a right part, plus one extra
  vertex. They have been obtained by carefully removing well-chosen maximal
  or minimal elements, until reaching this kind of configuration.
\end{remark}

\subsection{Weak order on the Coxeter group}

The cambrian lattice $\cam(W,c)$ is known to be a quotient of the weak
order on $W$ such that the fibers of the quotient are intervals in the
weak order. By Proposition \ref{contraction}, when the rank of $W$ is
at least $3$ and $W$ is not $\TA_1\times \TA_1 \times \TA_1$, the weak
order poset on $W$ has a wild representation type. When
$W=\TA_1\times \TA_1 \times \TA_1$, then the weak order poset is
isomorphic to the corresponding cambrian lattice, which is a cube. In
particular, it has a tame representation type. Finally, it is easy to
check that when the rank of $W$ is at most $2$, then the weak order
poset has a finite representation type. Indeed, in these cases, the
posets are obtained by adding a minimal and a maximal element to the
disjoint union of two chains. These posets are known to be of finite
representation type (see for instance \cite{chaptal}).

\section{Antichain posets}

This is another family of posets, expected but not known to be
derived equivalent to cambrian lattices.

Let $\Phi$ be a finite root system. The set $\Phi_+$ of positive roots
in $\Phi$ is endowed with a partial order by the relation
$\alpha \leq \beta$ if and only if $\beta - \alpha$ is a positive
linear combination of simple positive roots. This poset is called the
root poset of $\Phi$. If $n$ is the rank of the root system, this
poset has $n$ minimal elements, the simple positive roots.

Given a finite root system $\Phi$, the poset $\nest(\Phi)$ is defined
as the lattice of order ideals in the root poset of $\Phi$. Recall
that an order ideal in a poset $P$ is a subset $L$ of $P$ such that
$x \in L$ and $y \leq x$ implies $y \in L$. The partial order on the
set of order ideals is given by inclusion.

For $S$ a subset of $\Phi_+$, let us denote by $L(S)$ the order ideal
generated by $S$. Every order ideal can be uniquely written as $L(S)$
where $S$ is an antichain in the root poset.

The poset $\nest(\Phi)$ has a unique minimal element, the empty order
ideal. It also has a unique maximal element, the full root
poset. Moreover the restricted poset on the set of order ideals
contained in $L(\alpha_1,\dots,\alpha_n)$ is isomorphic to the $n$-cube poset.

\subsection{Easy types}

If the rank $n \leq 2$, then every poset $\nest(\Phi)$ has a finite representation type. Indeed, in this case Hasse diagram of the poset is a commutative square with a (possibly empty) chain attached to the maximal element of the square. By Proposition \ref{flip_flop}, it is derived equivalent to a Dynkin diagram of type $\TD$. The result follows from Proposition \ref{domine}. 

If the rank $n \geq 4$, then $\nest(\Phi)$ contains a $4$-cube, hence
one can apply Lemma \ref{rank4} to get that representation type is wild.

There remains only to handle the cases of rank $3$.

\subsection{Types of rank $3$}

For the reducible type $(\TA_1)^3$, the poset $\nest(\Phi)$ is
isomorphic to the cube poset. By Lemma \ref{cube}, it has a tame
representation type.

Assume now that $\Phi$ is not of type $(\TA_1)^3$. Then there exist at
least two simple roots (say $\alpha_1$ and $\alpha_2$) in the root poset
that are covered by a common root $\beta$. Then $\beta$ covers only
$\alpha_1$ and $\alpha_2$. This property also holds for the substitute
of the root poset in type $\TyH_3$ introduced by D. Armstrong in
\cite[\S 5.4.1]{armstrong}.

The Hasse diagram of the poset $\nest(\Phi)$ contains a cube, as its
restriction to the order ideals contained in
$L(\alpha_1,\alpha_2,\alpha_3)$. The order ideal $L(\beta)$ covers
$L(\alpha_1,\alpha_2)$ and no other vertex of the cube.

Consider the induced poset $Q$ on the vertices $L(\alpha_i)$,
$L(\alpha_i, \alpha_j)$ and $L(\beta)$. This is just a quiver, with
no commuting relation. Removing $L(\beta)$ gives an affine quiver of
type $\TA_5^{(1)}$. Therefore $Q$ has wild representation type, and so
does $\nest(\Phi)$.

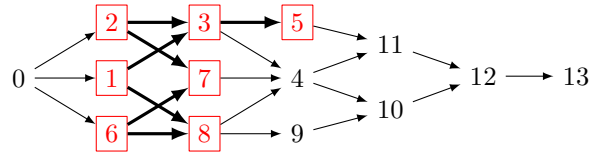
\begin{center}
  \begin{tikzpicture}[rotate=270,>=latex,line join=bevel,scale=0.7]
    \node (node_13) at (35.0bp,307.0bp) [draw,draw=none] {$13$}; \node
    (node_9) at (66.0bp,157.0bp) [draw,draw=none] {$9$}; \node
    (node_8) at (66.0bp,107.0bp) [draw,color=red] {$8$}; \node
    (node_7) at (36.0bp,107.0bp) [draw,color=red] {$7$}; \node
    (node_6) at (66.0bp,57.0bp) [draw,color=red] {$6$}; \node (node_5)
    at (6.0bp,157.0bp) [draw,color=red] {$5$}; \node (node_4) at
    (36.0bp,157.0bp) [draw,draw=none] {$4$}; \node (node_3) at
    (6.0bp,107.0bp) [draw,color=red] {$3$}; \node (node_2) at
    (6.0bp,57.0bp) [draw,color=red] {$2$}; \node (node_1) at
    (36.0bp,57.0bp) [draw,color=red] {$1$}; \node (node_0) at
    (36.0bp,7.0bp) [draw,draw=none] {$0$}; \node (node_11) at
    (18.0bp,207.0bp) [draw,draw=none] {$11$}; \node (node_10) at
    (53.0bp,207.0bp) [draw,draw=none] {$10$}; \node (node_12) at
    (35.0bp,257.0bp) [draw,draw=none] {$12$}; \draw [black,->,very
    thick] (node_2) ..controls (6.0bp,69.947bp) and (6.0bp,80.897bp)
    .. (node_3); \draw [black,->] (node_0) ..controls
    (28.309bp,20.305bp) and (20.92bp,32.128bp) .. (node_2); \draw
    [black,->] (node_3) ..controls (13.691bp,120.31bp) and
    (21.08bp,132.13bp) .. (node_4); \draw [black,->] (node_10)
    ..controls (48.466bp,220.09bp) and (44.23bp,231.39bp)
    .. (node_12); \draw [black,->] (node_4) ..controls
    (40.282bp,170.09bp) and (44.283bp,181.39bp) .. (node_10); \draw
    [black,->,very thick] (node_6) ..controls (58.309bp,70.305bp) and
    (50.92bp,82.128bp) .. (node_7); \draw [black,->] (node_5)
    ..controls (9.0047bp,170.02bp) and (11.785bp,181.14bp)
    .. (node_11); \draw [black,->,very thick] (node_1) ..controls
    (43.691bp,70.305bp) and (51.08bp,82.128bp) .. (node_8); \draw
    [black,->] (node_12) ..controls (35.0bp,269.95bp) and
    (35.0bp,280.9bp) .. (node_13); \draw [black,->] (node_11)
    ..controls (22.282bp,220.09bp) and (26.283bp,231.39bp)
    .. (node_12); \draw [black,->,very thick] (node_1) ..controls
    (28.309bp,70.305bp) and (20.92bp,82.128bp) .. (node_3); \draw
    [black,->] (node_4) ..controls (31.466bp,170.09bp) and
    (27.23bp,181.39bp) .. (node_11); \draw [black,->] (node_8)
    ..controls (66.0bp,119.95bp) and (66.0bp,130.9bp) .. (node_9);
    \draw [black,->,very thick] (node_2) ..controls
    (13.691bp,70.305bp) and (21.08bp,82.128bp) .. (node_7); \draw
    [black,->] (node_0) ..controls (43.691bp,20.305bp) and
    (51.08bp,32.128bp) .. (node_6); \draw [black,->,very thick]
    (node_3) ..controls (6.0bp,119.95bp) and (6.0bp,130.9bp)
    .. (node_5); \draw [black,->] (node_0) ..controls
    (36.0bp,19.947bp) and (36.0bp,30.897bp) .. (node_1); \draw
    [black,->] (node_8) ..controls (58.309bp,120.31bp) and
    (50.92bp,132.13bp) .. (node_4); \draw [black,->] (node_9)
    ..controls (62.745bp,170.02bp) and (59.733bp,181.14bp)
    .. (node_10); \draw [black,->] (node_7) ..controls
    (36.0bp,119.95bp) and (36.0bp,130.9bp) .. (node_4); \draw
    [black,->,very thick] (node_6) ..controls (66.0bp,69.947bp) and
    (66.0bp,80.897bp) .. (node_8);
  \end{tikzpicture}
\end{center}

\section{Stokes lattices}

Let us turn to another family of posets, containing the cambrian
lattices of type $\TA$. The Stokes lattices were introduced by the
first author as posets in \cite{chapoton} (inspired by previous work of
Baryshnikov in \cite{bary}) and they were proved to be lattices by Garver
and McConville in \cite{garver}.

Since the article \cite{garver}, the Stokes lattices can even be
considered as a special case of a more general class of lattices,
attached to dissections of polygons. Our arguments below work just the
same for this extended class, which has very similar properties.

Just as the cambrian lattices, the Stokes lattices also have
$n$-regular Hasse diagrams, where $n$ is a parameter called the rank
of the Stokes lattice.

\subsection{Easy types}

There are essentially the same easy cases as in the Cambrian setting.

In rank $n \leq 2$, the only possible Stokes posets are cambrian
lattices of type $\TA$, already considered before.

In rank $n \geq 4$, one can also use Lemma \ref{rank4} to get wild
representation type.

So once again, there remains to handle the case of rank $3$.

\subsection{Stokes lattices of rank $3$}

Excluding the intersection with the cambrian cases of type $\TA$
(reducible or not), there remains only one case to consider, which is
a poset with 12 vertices and the following shape.

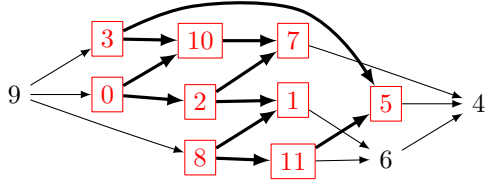
\begin{center}
\begin{tikzpicture}[rotate=270,>=latex,line join=bevel,scale=0.7]
\node (node_9) at (49.647bp,7.0bp) [draw,draw=none] {$9$};
  \node (node_8) at (83.647bp,107.0bp) [draw,color=red] {$8$};
  \node (node_7) at (20.647bp,157.0bp) [draw,color=red] {$7$};
  \node (node_6) at (84.647bp,207.0bp) [draw,draw=none] {$6$};
  \node (node_5) at (54.647bp,207.0bp) [draw,color=red] {$5$};
  \node (node_4) at (54.647bp,257.0bp) [draw,draw=none] {$4$};
  \node (node_3) at (19.647bp,57.0bp) [draw,color=red] {$3$};
  \node (node_2) at (53.647bp,107.0bp) [draw,color=red] {$2$};
  \node (node_1) at (52.647bp,157.0bp) [draw,color=red] {$1$};
  \node (node_10) at (20.647bp,107.0bp) [draw,color=red] {$10$};
  \node (node_11) at (85.647bp,157.0bp) [draw,color=red] {$11$};
  \node (node_0) at (49.647bp,57.0bp) [draw,color=red] {$0$};
  \draw [black,->,very thick] (node_3) ..controls (19.896bp,69.947bp) and (20.124bp,80.897bp)  .. (node_10);
  \draw [black,->,very thick] (node_0) ..controls (50.643bp,69.947bp) and (51.555bp,80.897bp)  .. (node_2);
  \draw [black,->] (node_5) ..controls (54.647bp,219.95bp) and (54.647bp,230.9bp)  .. (node_4);
  \draw [black,->,very thick] (node_11) ..controls (77.7bp,170.31bp) and (70.064bp,182.13bp)  .. (node_5);
  \draw [black,->,very thick] (node_2) ..controls (53.398bp,119.95bp) and (53.17bp,130.9bp)  .. (node_1);
  \draw [black,->] (node_11) ..controls (85.398bp,169.95bp) and (85.17bp,180.9bp)  .. (node_6);
  \draw [black,->] (node_7) ..controls (25.874bp,173.65bp) and (33.233bp,195.58bp)  .. (39.647bp,214.0bp) .. controls (42.727bp,222.85bp) and (46.246bp,232.71bp)  .. (node_4);
  \draw [black,->,very thick] (node_8) ..controls (75.7bp,120.31bp) and (68.064bp,132.13bp)  .. (node_1);
  \draw [black,->,very thick] (node_8) ..controls (84.145bp,119.95bp) and (84.601bp,130.9bp)  .. (node_11);
  \draw [black,->] (node_6) ..controls (76.956bp,220.31bp) and (69.567bp,232.13bp)  .. (node_4);
  \draw [black,->,very thick] (node_10) ..controls (20.647bp,119.95bp) and (20.647bp,130.9bp)  .. (node_7);
  \draw [black,->,very thick] (node_2) ..controls (45.138bp,120.38bp) and (36.888bp,132.38bp)  .. (node_7);
  \draw [black,->] (node_9) ..controls (41.956bp,20.305bp) and (34.567bp,32.128bp)  .. (node_3);
  \draw [black,->,very thick] (node_3) ..controls (12.336bp,71.655bp) and (5.2931bp,86.41bp)  .. (2.6471bp,100.0bp) .. controls (-0.57297bp,116.54bp) and (-1.5011bp,149.61bp)  .. (5.6471bp,164.0bp) .. controls (12.972bp,178.75bp) and (28.034bp,190.4bp)  .. (node_5);
  \draw [black,->] (node_1) ..controls (60.898bp,170.38bp) and (68.899bp,182.38bp)  .. (node_6);
  \draw [black,->] (node_9) ..controls (49.647bp,19.947bp) and (49.647bp,30.897bp)  .. (node_0);
  \draw [black,->,very thick] (node_0) ..controls (42.213bp,70.305bp) and (35.07bp,82.128bp)  .. (node_10);
  \draw [black,->] (node_9) ..controls (54.634bp,21.763bp) and (60.115bp,36.984bp)  .. (64.647bp,50.0bp) .. controls (69.458bp,63.816bp) and (74.8bp,79.603bp)  .. (node_8);
\end{tikzpicture}
\end{center}
By Proposition \ref{wild_subposet}, the subquiver on vertices
$\{0,1,2,3,5,7,8,10,11\}$ and its commutative square $(0,2,10,7)$ is of wild representation type. So, the poset is of wild representation type.


{Fr\'ed\'eric Chapoton} \\
{Institut de Recherche Math\'ematique Avanc\'ee, CNRS UMR 7501, Universit\'e de Strasbourg, F-67084 Strasbourg Cedex, France} \\
{chapoton@unistra.fr}\\
{Baptiste Rognerud} \\
{Institut de Recherche Math\'ematique Avanc\'ee, CNRS UMR 7501, Universit\'e de Strasbourg, F-67084 Strasbourg Cedex, France} \\
{rognerud@unistra.fr}

\end{document}